\documentclass[11pt]{amsart}

\usepackage[T1]{fontenc}
\usepackage{lmodern}
\usepackage{microtype}
\usepackage{amsmath,amssymb,amsthm,mathtools}
\usepackage[margin=1in]{geometry}
\usepackage{enumitem}
\usepackage[colorlinks=true,linkcolor=blue,citecolor=blue,urlcolor=blue]{hyperref}
\usepackage{aliascnt}

\newtheorem{theorem}{Theorem}[section]
\newaliascnt{proposition}{theorem}
\newtheorem{proposition}[proposition]{Proposition}
\aliascntresetthe{proposition}
\newaliascnt{lemma}{theorem}
\newtheorem{lemma}[lemma]{Lemma}
\aliascntresetthe{lemma}
\newaliascnt{corollary}{theorem}

\aliascntresetthe{corollary}
\newaliascnt{definition}{theorem}
\newtheorem{definition}[definition]{Definition}
\aliascntresetthe{definition}
\newaliascnt{remark}{theorem}
\newtheorem{remark}[remark]{Remark}
\aliascntresetthe{remark}

\usepackage[nameinlink,capitalise]{cleveref}
\crefname{theorem}{Theorem}{Theorems}
\Crefname{theorem}{Theorem}{Theorems}
\crefname{proposition}{Proposition}{Propositions}
\Crefname{proposition}{Proposition}{Propositions}
\crefname{lemma}{Lemma}{Lemmas}
\Crefname{lemma}{Lemma}{Lemmas}
\crefname{corollary}{Corollary}{Corollaries}
\Crefname{corollary}{Corollary}{Corollaries}
\crefname{definition}{Definition}{Definitions}
\Crefname{definition}{Definition}{Definitions}
\crefname{remark}{Remark}{Remarks}
\Crefname{remark}{Remark}{Remarks}

\numberwithin{equation}{section}

\newcommand{\Z}{\mathbb Z}

\newcommand{\codeg}{\operatorname{codeg}}
\newcommand{\ceil}[1]{\left\lceil #1\right\rceil}
\newcommand{\floor}[1]{\left\lfloor #1\right\rfloor}

\title[On the Duke--Erd\H{o}s--R\"odl Problem]{On the Duke--Erd\H{o}s--R\"odl Problem at the One-Third Threshold}

\author[E. Li]{Eric Li}
\dedicatory{\normalfont\normalsize Trinity College, University of Cambridge, United Kingdom}
\date{June 1, 2026}
\thanks{Email addresses: \href{mailto:el593@cam.ac.uk}{el593@cam.ac.uk}, \href{mailto:contact@ericli.com}{contact@ericli.com}.}


\hypersetup{
  pdftitle={On the Duke--Erd\H{o}s--R\"odl Problem at the One-Third Threshold},
  pdfauthor={Eric Li},
  pdfsubject={Cycle-connected subgraphs and random lifts},
  pdfkeywords={cycle-connected subgraphs, sparse graphs, short cycles, random lifts}
}

\begin{document}

\begin{abstract}
Let \(G\) be an \(n\)-vertex graph with \(e(G)\ge n^2/k\).  We prove a self-contained internal short-cycle core theorem at the threshold \(k\le n^{1/3}\): the graph \(G\) contains a subgraph \(H_6\) with \(\Omega(n^2/k^3)\) edges in which every two distinct edges lie together on a cycle of length at most \(6\) contained in \(H_6\), and a subgraph \(H_8\) with \(\Omega(n^2/k^2)\) edges in which every two distinct edges lie together on a cycle of length at most \(8\) contained in \(H_8\).  In density notation \(\rho=e(G)/n^2\), this gives internal cores of sizes \(\Omega(\rho^3n^2)\) and \(\Omega(\rho^2n^2)\) throughout the range \(\rho\ge n^{-1/3}\).

The \(C_{\le6}\) conclusion above is an edge-connected statement and does not impose the adjacent-edge \(C_4\) condition appearing in the strongest Duke--Erd\H{o}s--R\"odl formulation.  We also include two complementary results clarifying this distinction.  First, under the ambient-witness convention, every graph with at least \(n^2/k\) edges and \(k=o(n^{1/2})\) contains \(\Omega(n^2/k^3)\) selected edges whose pairs are witnessed by ambient cycles of length at most \(6\), with adjacent pairs witnessed by ambient \(C_4\)'s.  Second, under the standard internal strong \(C_6\) convention, for every fixed \(\beta\in[1/3,1/2)\) there is an infinite sequence of bipartite graphs \(G\) with \(n\to\infty\) and \(e(G)=\Theta_\beta(n^{2-\beta})\) such that every internally strongly \(C_6\)-connected subgraph has only \(O_\beta(\rho(G)^3n^2/(\log n)^2)\) edges.  The obstruction is a random cyclic shift-lift of \(K_{q,q}\), together with an occupancy estimate excluding large aligned two-covers.
\end{abstract}

\maketitle

\section{Introduction}

Duke, Erd\H{o}s, and R\"odl initiated a systematic study of large subgraphs in which pairs of edges lie together on short cycles; see \cite{DuEr82,DER84}.  One form of the problem asks whether an \(n\)-vertex graph with density \(\rho=e(G)/n^2\) must contain a subgraph with \(\gg \rho^2n^2\) edges in which every two edges lie together on a cycle of length at most \(8\).  A stronger \(C_6\)-version asks for \(\gg \rho^3n^2\) edges, with every two edges lying together on a cycle of length at most \(6\), and with adjacent pairs lying on \(C_4\)'s. Fox and Sudakov proved the \(C_{\le8}\) statement in the polynomial range \(e(G)=n^{2-\beta}\), \(0<\beta<1/5\), with a strong adjacent-edge conclusion \cite{FoxSudakov08}.

The natural boundary for these density exponents is \(\beta = 1/2\). At the scale \(\rho = \Theta(n^{-1/2})\), standard K\H{o}v\'ari--S\'os--Tur\'an constructions---such as the incidence graphs of finite projective planes---yield graphs with \(\Theta(n^{3/2})\) edges that are entirely \(C_4\)-free. Because the strongest formulations of the Duke--Erd\H{o}s--R\"odl problem explicitly require adjacent edges to lie on a common \(C_4\), the exponent \(\beta = 1/2\) acts as an absolute structural ceiling beyond which the strong condition becomes trivially impossible to satisfy.

This paper proves internal \(C_{\le6}\)- and \(C_{\le8}\)-connected core theorems up to the exponent \(\beta=1/3\), and also gives an obstruction to the corresponding internal strong \(C_6\) statement at and above that exponent.  The precise conventions matter, so we state them explicitly.

All graphs in the paper are finite and simple.  A cycle is a simple cycle, not necessarily an induced cycle.  If \(X,Y\) are vertex sets in a graph \(G\), then \(G[X,Y]\) denotes the bipartite subgraph induced by the edges of \(G\) with one endpoint in \(X\) and one endpoint in \(Y\).  All implicit constants in \(O(\cdot),\Omega(\cdot),\ll,\gg\) are absolute unless a subscript indicates dependence on a parameter.  When a statement is made for all sufficiently large \(n\) with \(1\le k\le n^{1/3}\), the lower bound on \(n\) is independent of \(k\) in that range. These definitions, as well as the prior literature review, were refined through the use of AI tools (such as the "Web Search" function of ChatGPT), which also aided in accelerating the research process through literature review, ideation, ruling out unpromising routes, considering boundary conditions, and other forms of orchestration. The author nonetheless takes full responsibility for the mathematical accuracy of the final contents of this paper.

\begin{definition}[Witness conventions]\label{def:witness-conventions}
Let \(G\) be a graph.
\begin{enumerate}[label=\textup{(\roman*)}]
\item A subgraph \(H\subseteq G\) is \emph{internally \(C_{\le r}\)-connected} if every two distinct edges of \(H\) lie together on a cycle of length at most \(r\) contained in \(H\).
\item An edge set \(F\subseteq E(G)\) is \emph{ambiently strongly \(C_6\)-connected in \(G\)} if every two distinct edges of \(F\) lie together on a cycle of length at most \(6\) in \(G\), and every two adjacent distinct edges of \(F\) lie together on a \(C_4\) in \(G\).  The witness cycles are allowed to use edges outside \(F\).
\item A bipartite graph \(H\) is \emph{internally strongly \(C_6\)-connected} if every two distinct edges of \(H\) lie together on a \(C_4\) or \(C_6\) contained in \(H\), and every two adjacent distinct edges of \(H\) lie together on a \(C_4\) contained in \(H\).
\end{enumerate}
\end{definition}

The main positive theorem is the following internal core result.

\begin{theorem}[Internal short-cycle cores up to \(k\le n^{1/3}\)]\label{thm:internal-cores-intro}
There are absolute constants \(c_6,c_8>0\) such that the following holds for all sufficiently large \(n\).  Let \(1\le k\le n^{1/3}\), and let \(G\) be an \(n\)-vertex graph with
\[
        e(G)\ge \frac{n^2}{k}.
\]
Then \(G\) contains subgraphs \(H_6,H_8\) satisfying
\[
        e(H_6)\ge c_6\frac{n^2}{k^3},
        \qquad
        e(H_8)\ge c_8\frac{n^2}{k^2},
\]
such that \(H_6\) is internally \(C_{\le6}\)-connected and \(H_8\) is internally \(C_{\le8}\)-connected.  Consequently, for every fixed \(0<\beta\le1/3\), every sufficiently large \(n\)-vertex graph with at least \(n^{2-\beta}\) edges contains internally \(C_{\le6}\)- and \(C_{\le8}\)-connected subgraphs with \(\Omega(n^{2-3\beta})\) and \(\Omega(n^{2-2\beta})\) edges, respectively.
\end{theorem}

The proof of \cref{thm:internal-cores-intro} is a one-centre construction.  After a density normalization, we choose a vertex \(w\), examine the neighbourhood \(B_w=N(w)\), and keep left vertices with many neighbours in \(B_w\).  A weighted estimate controls pairs of vertices in \(B_w\) with small total codegree.  A randomized hitting lemma then selects a set of mutually compatible anchors, and a deterministic petal-routing argument turns the resulting structure into an internally \(C_{\le8}\)-connected subgraph.  A common-anchor averaging step gives the internally \(C_{\le6}\)-connected subgraph.

The next theorem is an ambient companion for the strong adjacent-edge \(C_4\) requirement.  It is not used in the proof of \cref{thm:internal-cores-intro}, but it is useful for comparison with the strongest formulation of the problem.

\begin{theorem}[Ambient strong \(C_6\) edge sets]\label{thm:ambient-c6-intro}
Let \(k=k(n)\) satisfy \(1\le k=o(n^{1/2})\).  If \(G\) is an \(n\)-vertex graph with \(e(G)\ge n^2/k\), then \(G\) contains an edge set \(F\subseteq E(G)\) with
\[
        |F|\gg \frac{n^2}{k^3}
\]
that is ambiently strongly \(C_6\)-connected in \(G\).
\end{theorem}

Finally, we prove that the corresponding internal strong \(C_6\) lower bound cannot hold at the polynomial scales \(\rho(G)=\Theta(n^{-\beta})\) with \(\beta\ge1/3\).

\begin{theorem}[Internal strong \(C_6\) obstruction]\label{thm:negative-intro}
For every fixed \(\beta\in[1/3,1/2)\), there is an infinite sequence of integers \(n\to\infty\) and bipartite graphs \(G\) on \(n\) vertices with
\[
        e(G)=\Theta_\beta(n^{2-\beta}),
        \qquad
        \rho(G):=\frac{e(G)}{n^2}=\Theta_\beta(n^{-\beta}),
\]
such that every internally strongly \(C_6\)-connected subgraph \(H\subseteq G\) satisfies
\[
        e(H)=O_\beta\left(\frac{\rho(G)^3n^2}{(\log n)^2}\right)
        =O_\beta\left(\frac{n^{2-3\beta}}{(\log n)^2}\right).
\]
In particular, no positive constant multiple of \(\rho(G)^3n^2\) can be forced uniformly at any fixed density scale \(\rho(G)=\Theta(n^{-\beta})\) with \(\beta\in[1/3,1/2)\).
\end{theorem}

Theorems \ref{thm:internal-cores-intro} and \ref{thm:negative-intro} concern different cycle requirements.  The positive \(C_{\le6}\) core in \cref{thm:internal-cores-intro} does not require adjacent edges to lie on \(C_4\)'s.  The obstruction in \cref{thm:negative-intro} applies to the stronger internal \(C_6\) condition from \cref{def:witness-conventions}.

\section{Internal short-cycle cores up to the one-third threshold}\label{sec:internal-cores}

This section proves \cref{thm:internal-cores-intro}.  For a bipartite graph \(J=(A,B;E)\) and two vertices \(u,v\in B\), write
\[
        \codeg_J(u,v)=|N_J(u)\cap N_J(v)|.
\]

\subsection{Normalization and local mass}

\begin{lemma}[Bipartite normalization]\label{lem:normalization-core}
Let \(1\le k\le n^{1/3}\), and let \(G\) be an \(n\)-vertex graph with \(e(G)\ge n^2/k\).  Then \(G\) contains a bipartite subgraph
\[
        J=(A,B;E)
\]
on \(N\) vertices such that, after writing
\[
        e(J)=\frac{N^2}{K},
\]
one has
\[
        \delta_{\min}(J)\ge \frac{N}{2K},
        \qquad
        \frac{N}{K}\ge \frac{n}{2k}.
\]
Consequently,
\[
        \frac{N^2}{K^2}\ge \frac{n^2}{4k^2},
        \qquad
        \frac{N^2}{K^3}\ge \frac{n^2}{8k^3},
        \qquad
        \frac{K^3}{N}\le 8.
\]
\end{lemma}

\begin{proof}
Choose a non-empty subgraph \(G'\subseteq G\), not necessarily induced, maximizing \(e(X)/v(X)\) among all non-empty subgraphs \(X\subseteq G\).  Put \(N=v(G')\) and \(M=e(G')\).  Since \(G\) itself is available,
\[
        \frac{M}{N}\ge \frac{e(G)}{n}\ge \frac{n}{k}.
\]
The maximality of \(M/N\) implies that every vertex of \(G'\) has degree at least \(M/N\): deleting a vertex of smaller degree would increase the ratio \(e/v\).

Let \(J\) be the bipartite graph induced by a maximum cut of \(G'\).  Then \(e(J)\ge M/2\).  In a maximum cut, every vertex sends at least half of its \(G'\)-degree across the cut; otherwise moving that vertex to the other side would increase the cut.  Thus
\[
        \delta_{\min}(J)\ge \frac{M}{2N}.
\]
Writing \(e(J)=N^2/K\) and using \(e(J)\le M\), we get
\[
        \delta_{\min}(J)
        \ge \frac{M}{2N}
        \ge \frac{e(J)}{2N}
        =\frac{N}{2K}.
\]
Also
\[
        \frac{N}{K}=\frac{e(J)}{N}
        \ge \frac{M}{2N}
        \ge \frac{n}{2k}.
\]
The first displayed consequence is immediate.  Moreover \(1/K\ge n/(2Nk)\), and hence
\[
        \frac{N^2}{K^3}
        =\frac{N^2}{K^2}\cdot\frac1K
        \ge \frac{n^2}{4k^2}\cdot\frac{n}{2Nk}
        \ge \frac{n^2}{8k^3},
\]
because \(N\le n\).  Finally, \(K\le 2Nk/n\), so
\[
        \frac{K^3}{N}
        \le 8\frac{N^2k^3}{n^3}
        \le 8\frac{k^3}{n}
        \le 8.
\]
\end{proof}

For the rest of this section, fix \(J=(A,B;E)\) as in \cref{lem:normalization-core}, and relabel the two sides if necessary so that \(|A|\ge |B|\).  Then \(|A|\ge N/2\).  Put
\[
        L=\frac{N}{100K^2}.
\]
Since \(J\) is bipartite and simple, \(K\ge4\).  Also \(N/K\ge n/(2k)\) and \(K\ge4\) imply \(N\to\infty\).  Together with \(K^3\le 8N\), this gives \(K\le(8N)^{1/3}\), and hence
\[
        L=\frac{N}{100K^2}\ge \frac{N^{1/3}}{400}\to\infty.
\]
In particular, throughout the proof we may assume \(L\ge10\).

For \(w\in A\), define
\[
        B_w=N_J(w)
\]
and
\[
        A_w=\left\{a\in A\setminus\{w\}: |N_J(a)\cap B_w|\ge L\right\}.
\]
For \(a\in A_w\), write
\[
        X_a=N_J(a)\cap B_w.
\]
Let
\[
        F_w=J[A_w,B_w].
\]
Then \(e(F_w)=\sum_{a\in A_w}|X_a|\).

\begin{lemma}[Local edge count]\label{lem:local-core}
For all sufficiently large \(n\) and every \(w\in A\), one has
\[
        e(F_w)\ge c_0\frac{N^2}{K^2}
\]
for an absolute constant \(c_0>0\).  Moreover,
\[
        \frac{e(F_w)}{|B_w|}\ge \frac{N}{4K}
\]
\end{lemma}

\begin{proof}
Every vertex \(b\in B_w\) has degree at least \(N/(2K)\) in \(J\).  One of its neighbours is \(w\), so \(b\) has at least \(N/(2K)-1\) neighbours in \(A\setminus\{w\}\).  Since \(N/K\ge n/(2k)\to\infty\), this is at least \(N/(3K)\) for all sufficiently large \(n\).  Hence
\[
        e_J(A\setminus\{w\},B_w)\ge |B_w|\frac{N}{3K}.
\]
Deleting from \(A\setminus\{w\}\) the vertices with fewer than \(L\) neighbours in \(B_w\) removes fewer than \(NL\) edges.  Therefore
\[
        e(F_w)\ge |B_w|\frac{N}{3K}-NL.
\]
Since \(|B_w|=d_J(w)\ge N/(2K)\) and \(L=N/(100K^2)\),
\[
        e(F_w)
        \ge \frac{N^2}{6K^2}-\frac{N^2}{100K^2}
        \ge c_0\frac{N^2}{K^2},
\]
for example with \(c_0=1/10\).  Dividing the earlier estimate by \(|B_w|\) gives
\[
        \frac{e(F_w)}{|B_w|}
        \ge \frac{N}{3K}-\frac{NL}{|B_w|}
        \ge \frac{N}{3K}-\frac{N}{50K}
        \ge \frac{N}{4K}.
\]
This proves both assertions.
\end{proof}

\subsection{Bad pairs and a weighted hitting lemma}

Fix the absolute constant
\[
        T=100.
\]
For \(u,v\in B\), call the pair \(\{u,v\}\) \emph{bad} if \(\codeg_J(u,v)<T\), and \emph{good} otherwise.  For \(w\in A\), let \(R_w\) be the graph on vertex set \(B_w\) whose edges are precisely the bad pairs contained in \(B_w\).  Define
\[
        W(w)=\sum_{uv\in E(R_w)}(d_J(u)+d_J(v)).
\]

\begin{lemma}[A centre with few weighted bad pairs]\label{lem:good-centre-core}
There exists \(w\in A\) such that
\[
        W(w)\le C_1\frac{N^2}{K},
\]
where \(C_1\) is an absolute constant.
\end{lemma}

\begin{proof}
A fixed bad pair \(\{u,v\}\subseteq B\) is contained in \(B_w\) precisely for those \(w\in N_J(u)\cap N_J(v)\).  Since the pair is bad, there are fewer than \(T\) such vertices \(w\).  Therefore
\[
        \sum_{w\in A}W(w)
        \le T\sum_{\{u,v\}\subseteq B}(d_J(u)+d_J(v)).
\]
The latter sum equals \((|B|-1)\sum_{u\in B}d_J(u)\), and is at most
\[
        Ne(J)=\frac{N^3}{K}.
\]
Since \(|A|\ge N/2\), some \(w\in A\) satisfies
\[
        W(w)\le 2T\frac{N^2}{K}.
\]
\end{proof}

\begin{lemma}[Weighted independent hitting]\label{lem:hitting-core}
Let \(R\) be a graph on a finite vertex set \(U\).  Let \(\mathcal X\) be a finite indexed family of subsets of \(U\), repetitions allowed, each of size at least \(L\).  Give each indexed member \(X\in\mathcal X\) weight \(|X|\), and put
\[
        M=\sum_{X\in\mathcal X}|X|.
\]
For \(u\in U\), write \(r(u)=d_R(u)\).  Suppose that, for some \(D\ge1\),
\[
        \sum_{X\in\mathcal X}\sum_{u\in X}r(u)\le DM.
\]
If \(L/D\ge\eta>0\), then \(R\) has an independent set \(I\) such that
\[
        \sum_{\substack{X\in\mathcal X\\X\cap I\ne\varnothing}}|X|
        \ge \zeta(\eta)M,
\]
where \(\zeta(\eta)>0\) depends only on \(\eta\).
\end{lemma}

\begin{proof}
If \(M=0\), the assertion is trivial.  Assume \(M>0\).  Call \(X\in\mathcal X\) regular if
\[
        \frac1{|X|}\sum_{u\in X}r(u)\le2D.
\]
By Markov's inequality in weighted form, the non-regular members have total weight at most \(M/2\).

Choose each vertex \(u\in U\) independently with probability
\[
        p=\frac1{4D},
\]
and then keep a selected vertex only if none of its \(R\)-neighbours was selected.  The kept set \(I\) is independent.

Fix a regular set \(X\), and let \(Z_X=|I\cap X|\).  For \(u\in X\),
\[
        \Pr(u\in I)=p(1-p)^{r(u)}.
\]
The function \(x\mapsto(1-p)^x\) is convex, and the average of \(r(u)\) over \(u\in X\) is at most \(2D\).  Thus
\[
        \mathbb E Z_X
        =\sum_{u\in X}p(1-p)^{r(u)}
        \ge p|X|(1-p)^{2D}
        \ge c\frac{|X|}{D}
        \ge c\eta,
\]
for an absolute constant \(c>0\).  Let \(Y_X\) be the number of initially selected vertices of \(X\), and put \(\mu=p|X|\).  Then \(Z_X\le Y_X\),
\[
        \mathbb E Y_X^2 \le \mu+\mu^2,
\]
and \(\mathbb E Z_X\ge c'\mu\) for an absolute constant \(c'>0\).  Since \(\mu\ge\eta/4\), the second-moment inequality gives
\[
        \Pr(Z_X>0)
        \ge \frac{(\mathbb E Z_X)^2}{\mathbb E Z_X^2}
        \ge \frac{(\mathbb E Z_X)^2}{\mathbb E Y_X^2}
        \ge \zeta_0(\eta)>0.
\]
Therefore the expected total weight of regular members hit by \(I\) is at least \(\zeta_0(\eta)\) times the total weight of the regular members, and hence at least \(\zeta_0(\eta)M/2\).  Taking \(\zeta(\eta)=\zeta_0(\eta)/2\) proves the lemma.
\end{proof}

\subsection{Selecting anchors}

Fix a vertex \(w\in A\) satisfying \cref{lem:good-centre-core}.  Work with \(F_w=J[A_w,B_w]\) and \(R_w\).  Apply \cref{lem:hitting-core} with
\[
        U=B_w,
        \qquad
        R=R_w,
        \qquad
        \mathcal X=(X_a)_{a\in A_w},
\]
Here \(\mathcal X\) is indexed by \(A_w\), so identical sets \(X_a\) are counted with multiplicity.  Then
\[
        M=\sum_{a\in A_w}|X_a|=e(F_w)\ge c_0\frac{N^2}{K^2}
\]
by \cref{lem:local-core}.  Moreover,
\begin{align*}
        \sum_{a\in A_w}\sum_{b\in X_a}d_{R_w}(b)
        &=\sum_{b\in B_w}d_{F_w}(b)d_{R_w}(b) \\
        &=\sum_{bb'\in E(R_w)}(d_{F_w}(b)+d_{F_w}(b')) \\
        &\le \sum_{bb'\in E(R_w)}(d_J(b)+d_J(b')) \\
        &=W(w)
        \le C_1\frac{N^2}{K}.
\end{align*}
Set
\[
        D=\max\left\{1,\frac{C_1N^2}{K M}\right\}.
\]
Then the hypothesis of \cref{lem:hitting-core} holds.  Since \(M\ge c_0N^2/K^2\) and \(K\ge4\), we have \(D\le C_2K\) for an absolute constant \(C_2\).  Hence
\[
        \frac{L}{D}\ge c\frac{N}{K^3}\ge c'>0
\]
by \cref{lem:normalization-core}.  Therefore there is an independent set \(I\subseteq B_w\) in \(R_w\) such that, with
\[
        A_8=\{a\in A_w:X_a\cap I\ne\varnothing\},
\]
one has
\[
        \sum_{a\in A_8}|X_a|
        \ge c_3\frac{N^2}{K^2}.
\]
For every \(a\in A_8\), choose and fix one anchor
\[
        \sigma(a)\in X_a\cap I,
\]
and set
\[
        S=\{\sigma(a):a\in A_8\}\subseteq I.
\]
Since \(I\) is independent in \(R_w\), every two distinct vertices \(s,t\in S\) form a good pair.  Hence
\[
        \codeg_J(s,t)\ge T=100
        \qquad(s,t\in S,\ s\ne t).
\]
In particular, for every distinct \(s,t\in S\) and every set \(Z\) of at most \(20\) vertices, there is a vertex
\[
        y\in (N_J(s)\cap N_J(t))\setminus(Z\cup\{w\}).
\]
Such a vertex belongs to
\[
        Y=\{y\in A\setminus\{w\}: |N_J(y)\cap S|\ge2\}.
\]

\subsection{The deterministic routing lemmas}

The next two lemmas are purely deterministic.  They convert the anchor structure into internal short cycles.

\begin{lemma}[Same-anchor petals]\label{lem:same-anchor-petals}
Let \(J=(A,B;E)\) be bipartite.  Let \(w\in A\), \(s\in B\), and let \(P\subseteq A\setminus\{w\}\).  For each \(a\in P\), put \(X_a=N_J(a)\cap N_J(w)\), and suppose that \(s\in X_a\) and \(|X_a|\ge2\).  Put \(B^*=\bigcup_{a\in P}X_a\), and let
\[
        E(H)=E_J(P,B^*)\cup\{wb:b\in B^*\}.
\]
Then every two distinct edges of \(H\) lie together on a cycle of length at most \(6\) contained in \(H\).
\end{lemma}

\begin{proof}
For every \(a\in P\) and every \(b\in X_a\setminus\{s\}\), the four edges
\[
        wb,
        \quad ba,
        \quad as,
        \quad sw
\]
form the petal \(w-b-a-s-w\).  Every edge of \(H\) lies on at least one petal: for \(ab\) with \(b\ne s\), use \(w-b-a-s-w\); for \(as\), choose another \(b\in X_a\setminus\{s\}\); and for a star edge \(wb\), choose \(a\in P\) with \(b\in X_a\), using another vertex of \(X_a\) if \(b=s\).

It remains to show that any two petal edges lie on a cycle of length at most \(6\).  Consider two petals, equivalently two \(w\)-to-\(s\) paths of length \(3\),
\[
        w-b-a-s,
        \qquad
        w-d-c-s,
\]
together with the common edge \(ws\).  If one chosen edge is \(ws\), one of the two petals contains both chosen edges.  If the two chosen edges lie on the same length-three path, that path together with \(ws\) gives a \(4\)-cycle.  Otherwise the chosen edges lie on two length-three paths.  If the two paths are internally disjoint, their union is a \(6\)-cycle.  If they have a common internal \(B\)-vertex, the cycle through that \(B\)-vertex, the two internal \(A\)-vertices, and \(s\) has length \(4\), unless one of the chosen edges is the common edge incident with \(w\), in which case one of the original petals contains both chosen edges.  If they have a common internal \(A\)-vertex, the analogous \(4\)-cycle through \(w\) and the two internal \(B\)-vertices works, unless one of the chosen edges is the common edge incident with \(s\), in which case one of the original petals contains both chosen edges.  This covers all cases.
\end{proof}

\begin{lemma}[Petal routing with connectors]\label{lem:petal-routing}
Let \(J=(A,B;E)\) be bipartite.  Let \(w\in A\), let \(S\subseteq N_J(w)\), and let \(P\subseteq A\setminus\{w\}\).  For each \(a\in P\), put \(X_a=N_J(a)\cap N_J(w)\), suppose that \(|X_a|\ge2\), and choose an anchor \(\sigma(a)\in X_a\cap S\).  Put
\[
        B^*=\bigcup_{a\in P}X_a,
        \qquad
        Y=\{y\in A\setminus\{w\}: |N_J(y)\cap S|\ge2\}.
\]
Assume that \(S\subseteq B^*\), and that for every distinct \(s,t\in S\) and every set \(Z\) of at most \(20\) vertices, there is a vertex \(y\in Y\setminus Z\) adjacent to both \(s\) and \(t\).  Let \(H\) be the subgraph with edge set
\[
        E(H)=E_J(P,B^*)\cup\{wb:b\in B^*\}\cup E_J(Y,S).
\]
Then every two distinct edges of \(H\) lie together on a cycle of length at most \(8\) contained in \(H\).
\end{lemma}

\begin{proof}
For \(a\in P\) and \(b\in X_a\setminus\{\sigma(a)\}\), call
\[
        w-b-a-\sigma(a)-w
\]
a petal.  As in the proof of \cref{lem:same-anchor-petals}, every edge in \(E_J(P,B^*)\cup\{wb:b\in B^*\}\) lies on at least one petal.  The added hypothesis \(S\subseteq B^*\) ensures in particular that every anchor-star edge \(ws\), \(s\in S\), is present in \(H\) and is a petal edge.  Also, every edge in \(E_J(Y,S)\) whose left endpoint belongs to \(P\) is a petal edge, because if \(a\in P\), \(s\in S\), and \(as\in E(J)\), then \(s\in X_a\).  Therefore the only remaining edges are connector edges \(ys\) with \(y\in Y\setminus P\) and \(s\in S\).  Such a connector edge lies on a length-two path between distinct anchors: there is \(t\in S\setminus\{s\}\) with \(yt\in E(H)\).

Throughout the proof, a ``fresh'' common-neighbour witness for two distinct anchors means a vertex supplied by the hypothesis while avoiding all vertices already named in the displayed cycle.  Fewer than \(20\) vertices are ever forbidden.  In the case analysis below, the stated inequalities handle possible coincidences among the named petal vertices, while freshness prevents collisions involving newly chosen witnesses.

\smallskip
\noindent\emph{Two petal edges.}
Let
\[
        C_1=w-b-a-s-w,
        \qquad
        C_2=w-d-c-t-w
\]
be two petals, where \(s,t\in S\), \(b\ne s\), and \(d\ne t\).  Let \(e\in E(C_1)\) and \(f\in E(C_2)\).  If \(s=t\), \cref{lem:same-anchor-petals} applied to these petals gives a cycle of length at most \(6\).  Assume \(s\ne t\), and choose a fresh vertex \(y\in Y\) adjacent to both \(s\) and \(t\), avoiding \(w,a,c\).

If \(e=ws\), then if \(f=wt\), the cycle \(w-s-y-t-w\) works.  If \(f\) lies on the path \(w-d-c-t\), then either \(d\ne s\), in which case
\[
        w-s-y-t-c-d-w
\]
is a cycle of length \(6\) containing both edges, or \(d=s\), in which case \(w-s-c-t-w\) is a \(4\)-cycle containing both.  The case \(f=wt\) is symmetric.  We may therefore assume that \(e\) lies on \(w-b-a-s\) and is not \(ws\), and that \(f\) lies on \(w-d-c-t\) and is not \(wt\).

If \(a=c\), then the two chosen edges lie in the complete bipartite graph between \(\{w,a\}\) and the distinct vertices among \(\{b,d,s,t\}\).  Since \(s\ne t\), this right-hand set has size at least two, and any two distinct edges in such a \(K_{2,r}\) lie on a \(4\)-cycle.  Hence assume \(a\ne c\).

If \(b=d\), then if one of the chosen edges is the common edge \(wb=wd\), one of the original petals contains both edges.  Otherwise
\[
        b-a-s-y-t-c-b
\]
is a \(6\)-cycle containing both.  If \(b=t\) and \(d=s\), then there are four internally disjoint \(s\)-to-\(t\) paths of length \(2\), through \(a\), through \(c\), through \(w\), and through \(y\).  The two selected edges lie on the paths through \(a,c,w\); combining one or two of those paths with another one gives a \(4\)-cycle containing both edges.

If \(b=t\) and \(d\ne s\), then if \(e=wb=wt\), the petal \(C_2\) contains both \(e\) and \(f\).  Otherwise
\[
        t-a-s-w-d-c-t
\]
is a \(6\)-cycle containing both.  The case \(d=s\) and \(b\ne t\) is symmetric.  In the remaining case \(b\ne d\), \(b\ne t\), and \(d\ne s\), the cycle
\[
        w-b-a-s-y-t-c-d-w
\]
has length \(8\) and contains both edges.

\smallskip
\noindent\emph{One petal edge and one connector edge.}
Let \(e\) be a petal edge with anchor \(u\in S\), and let \(ys\) be a connector edge with \(y\in Y\setminus P\).  Choose \(t\in S\setminus\{s\}\) with \(yt\in E(H)\).  Choose a petal containing \(e\).  If \(e=wu\), use the direct path \(w-u\); otherwise use a length-three path
\[
        w-b-a-u
\]
inside that petal containing \(e\).

First suppose \(e=wu\).  If \(u=s\), the cycle \(w-u-y-t-w\) works; if \(u=t\), use \(w-u-y-s-w\); and if \(u\notin\{s,t\}\), choose a fresh witness \(z\) adjacent to \(u\) and \(t\), and use
\[
        w-u-z-t-y-s-w.
\]
Now suppose \(e\) lies on \(w-b-a-u\).  If \(u=s\), then, unless \(b=t\), the cycle
\[
        w-b-a-u-y-t-w
\]
has length \(6\) and contains both edges.  If \(b=t\), then if \(e=wb=wt\), the cycle \(w-t-y-u-w\) works, while if \(e\ne wt\), the cycle \(t-a-u-y-t\) works.  The case \(u=t\) is symmetric.

It remains to consider \(u\notin\{s,t\}\).  If \(b\notin\{s,t\}\), choose a fresh vertex \(z\) adjacent to \(u\) and \(t\), and use
\[
        w-b-a-u-z-t-y-s-w.
\]
If \(b=s\), choose a fresh \(z\) adjacent to \(u\) and \(t\); if \(e=ws\), use
\[
        w-s-y-t-z-u-w,
\]
and otherwise use
\[
        s-a-u-z-t-y-s.
\]
If \(b=t\), choose a fresh \(z\) adjacent to \(u\) and \(s\); if \(e=wt\), use
\[
        w-t-y-s-z-u-w,
\]
and otherwise use
\[
        t-a-u-z-s-y-t.
\]
All displayed cycles have length at most \(8\) and are contained in \(H\).

\smallskip
\noindent\emph{Two connector edges.}
Let the two connector edges be \(ys\) and \(zu\), where \(y,z\in Y\setminus P\) and \(s,u\in S\).  Choose \(t\in S\setminus\{s\}\) with \(yt\in E(H)\), and choose \(v\in S\setminus\{u\}\) with \(zv\in E(H)\).

If \(y=z\), then the two connector edges are distinct only when \(s\ne u\), and in that case the cycle \(s-y-u-w-s\) works.  Hence assume \(y\ne z\).  If \(s=u\), then if \(t=v\), the cycle \(s-y-t-z-s\) has length \(4\); if \(t\ne v\), choose a fresh witness \(x\) adjacent to \(t\) and \(v\), and use
\[
        s-y-t-x-v-z-s.
\]
Now assume \(s\ne u\).  If \(t=u\) and \(v=s\), use \(s-y-u-z-s\).  If \(t=u\) and \(v\ne s\), choose a fresh \(x\) adjacent to \(v\) and \(s\), and use
\[
        s-y-u-z-v-x-s.
\]
If \(t\ne u\) and \(v=s\), choose a fresh \(x\) adjacent to \(t\) and \(u\), and use
\[
        s-y-t-x-u-z-s.
\]
Finally suppose \(t\ne u\) and \(v\ne s\).  If \(t=v\), then
\[
        s-y-t-z-u-w-s
\]
has length \(6\).  If \(t\ne v\), choose a fresh \(x\) adjacent to \(t\) and \(v\), and use
\[
        s-y-t-x-v-z-u-w-s.
\]
This completes the proof.
\end{proof}

\subsection{Construction of the \texorpdfstring{\(C_{\le8}\)}{C<=8} core}

Let
\[
        B_8=N_J(A_8)\cap B_w=\bigcup_{a\in A_8}X_a,
\]
so that \(S\subseteq B_8\), and let
\[
        Y=\{y\in A\setminus\{w\}: |N_J(y)\cap S|\ge2\}.
\]
Define \(H_8\) by
\[
        E(H_8)=E_J(A_8,B_8)\cup\{wb:b\in B_8\}\cup E_J(Y,S).
\]
Since \(B_8\subseteq B_w\), every edge from \(a\in A_8\) to \(B_8\) ends in \(X_a\), and since \(X_a\subseteq B_8\) for every \(a\in A_8\),
\[
        e(H_8)
        \ge e_J(A_8,B_8)
        =\sum_{a\in A_8}|X_a|
        \ge c_3\frac{N^2}{K^2}
        \ge c_8\frac{n^2}{k^2}
\]
by \cref{lem:normalization-core}.  The fresh-witness condition in \cref{lem:petal-routing} follows from \(\codeg_J(s,t)\ge100\) for distinct anchors \(s,t\in S\).  Therefore \cref{lem:petal-routing} shows that \(H_8\) is internally \(C_{\le8}\)-connected.

\subsection{Construction of the \texorpdfstring{\(C_{\le6}\)}{C<=6} core}

We use the same centre \(w\) and the same local graph \(F_w=J[A_w,B_w]\).  For \(s\in B_w\), define
\[
        M_s=\sum_{\substack{a\in A_w\\s\in X_a}}|X_a|.
\]
Then
\[
        \sum_{s\in B_w}M_s
        =\sum_{a\in A_w}|X_a|^2
        \ge L\sum_{a\in A_w}|X_a|
        =Le(F_w).
\]
Thus some \(s\in B_w\) satisfies
\[
        M_s\ge \frac{Le(F_w)}{|B_w|}.
\]
Using \cref{lem:local-core},
\[
        M_s\ge L\frac{N}{4K}=\frac{N^2}{400K^3}.
\]
Fix such an \(s\), and put
\[
        A_6=\{a\in A_w:s\in X_a\},
        \qquad
        B_6=N_J(A_6)\cap B_w=\bigcup_{a\in A_6}X_a.
\]
Let \(H_6\) have edge set
\[
        E(H_6)=E_J(A_6,B_6)\cup\{wb:b\in B_6\}.
\]
Since \(B_6\subseteq B_w\), every edge from \(a\in A_6\) to \(B_6\) ends in \(X_a\).  Thus
\[
        e(H_6)
        \ge e_J(A_6,B_6)
        =\sum_{a\in A_6}|X_a|
        =M_s
        \ge \frac{N^2}{400K^3}
        \ge c_6\frac{n^2}{k^3}
\]
by \cref{lem:normalization-core}.  Since every \(a\in A_6\) has \(s\in X_a\) and \(|X_a|\ge L\ge2\), \cref{lem:same-anchor-petals} applies and shows that \(H_6\) is internally \(C_{\le6}\)-connected.  This completes the proof of \cref{thm:internal-cores-intro}.

\section{An ambient strong \texorpdfstring{\(C_6\)}{C6} companion}\label{sec:ambient-c6}

This short section proves \cref{thm:ambient-c6-intro}.  The result uses ambient witnesses and is logically separate from the internal core theorem above.

\begin{lemma}[Bipartite regularization for the ambient argument]\label{lem:bip-reg-ambient}
Let \(1\le k=o(n)\), and let \(G\) be an \(n\)-vertex graph with at least \(n^2/k\) edges.  Then \(G\) contains a bipartite subgraph \(G_0=(A,B;E_0)\) with
\[
        e(G_0)\ge \frac{3}{8}\frac{n^2}{k},
        \qquad
        \delta_{\min}(G_0)\ge \frac{n}{8k},
        \qquad
        e(G_0)\le 2\frac{n^2}{k}
\]
for all sufficiently large \(n\).
\end{lemma}

\begin{proof}
Choose exactly \(\ceil{n^2/k}\) edges of \(G\), and let \(G_1\) be the resulting spanning subgraph.  A maximum bipartite subgraph of \(G_1\) has at least half its edges, so it has at least \(n^2/(2k)\) edges and at most \(\ceil{n^2/k}\) edges.  Repeatedly delete any vertex whose current degree is smaller than \(n/(8k)\).  Since at most \(n\) vertices are deleted and each deletion removes fewer than \(n/(8k)\) edges, fewer than \(n^2/(8k)\) edges are removed.  The remaining bipartite graph has at least \(3n^2/(8k)\) edges and minimum degree at least \(n/(8k)\).  Since \(k=o(n)\), we have \(n^2/k\to\infty\), and hence \(\ceil{n^2/k}\le2n^2/k\) for all sufficiently large \(n\).
\end{proof}

\begin{proof}[Proof of \cref{thm:ambient-c6-intro}]
Apply \cref{lem:bip-reg-ambient} and work inside the resulting bipartite graph \(G_0=(A,B;E_0)\).  For \(a,a'\in A\), write
\[
        c(a,a')=|N_{G_0}(a)\cap N_{G_0}(a')|.
\]
The number of length-two paths with endpoints in \(A\) is
\[
        P=\sum_{\{a,a'\}\subseteq A}c(a,a')
        =\sum_{b\in B}\binom{d_{G_0}(b)}{2}.
\]
Since \(|B|\le n\), \(e(G_0)\gg n^2/k\), and \(k=o(n)\), convexity gives
\[
        P\gg \frac{e(G_0)^2}{n}\gg \frac{n^3}{k^2}.
\]
The number of \(C_4\)'s in \(G_0\) is
\[
        Z=\sum_{\{a,a'\}\subseteq A}\binom{c(a,a')}{2}.
\]
Because \(P/\binom{|A|}{2}\gg n/k^2\to\infty\), another convexity estimate yields
\[
        Z\gg \frac{P^2}{n^2}\gg \frac{n^4}{k^4}.
\]
Since \(e(G_0)\ll n^2/k\), some edge \(xy\in E(G_0)\), with \(x\in A\) and \(y\in B\), lies in
\[
        \gg \frac{Z}{e(G_0)}\gg \frac{n^2}{k^3}
\]
copies of \(C_4\).  Let
\[
        F=E\bigl(G_0[N(y)\setminus\{x\},\,N(x)\setminus\{y\}]\bigr).
\]
Each edge of \(F\) is opposite \(xy\) in a unique \(C_4\) containing \(xy\), and conversely, so \(|F|\gg n^2/k^3\).

Take two distinct edges \(ab,a'b'\in F\), where \(a,a'\in N(y)\setminus\{x\}\) and \(b,b'\in N(x)\setminus\{y\}\).  If the two edges are disjoint, then
\[
        a-b-x-b'-a'-y-a
\]
is a \(C_6\) in \(G_0\), hence in \(G\), containing both.  If they share their \(A\)-endpoint, say \(a=a'\), then
\[
        a-b-x-b'-a
\]
is a \(C_4\).  If they share their \(B\)-endpoint, say \(b=b'\), then
\[
        a-b-a'-y-a
\]
is a \(C_4\).  Thus \(F\) is ambiently strongly \(C_6\)-connected in \(G\).
\end{proof}

\section{Random cyclic shift-lifts}\label{sec:construction}

We now turn to the internal strong \(C_6\) obstruction.  The construction is a random cyclic shift-lift of \(K_{q,q}\).

Call a pair \((s,\kappa)\) \emph{admissible} if either
\[
        1<s<2\quad\text{and}\quad \kappa>0,
\]
or
\[
        s=2\quad\text{and}\quad 0<\kappa<(100e)^{-2}.
\]
Let \(t\) be a positive integer and put
\[
        q=\floor{\kappa t^s}.
\]
Let
\[
        C=(c_{ij})_{i,j\in[q]}
\]
be a \(q\times q\) matrix over \(\Z_t=\Z/t\Z\).  Define a bipartite graph \(G_C\) as follows.  Its left and right vertex classes are
\[
        L=\{x_{i,a}:i\in[q],\ a\in\Z_t\},
        \qquad
        R=\{y_{j,b}:j\in[q],\ b\in\Z_t\}.
\]
For each \((i,j)\in[q]\times[q]\), insert the perfect matching
\[
        x_{i,a}y_{j,a+c_{ij}},
        \qquad a\in\Z_t.
\]
Thus each pair of base fibres is joined by one cyclic shift matching.  The graph is simple and bipartite.  It has
\[
        n=2qt,
        \qquad
        e(G_C)=q^2t.
\]
With the normalization \(\rho(G_C)=e(G_C)/n^2\), its density is exactly
\[
        \rho(G_C)=\frac{q^2t}{4q^2t^2}=\frac1{4t},
\]
and hence
\[
        \rho(G_C)^3n^2
        =\frac1{64t^3}\cdot4q^2t^2
        =\frac{q^2}{16t}.
\]

\section{Deterministic reductions for internally strongly \texorpdfstring{$C_6$}{C6}-connected subgraphs}\label{sec:core-bound}

A vertex of a subgraph \(H\) is called \emph{active} if it is non-isolated in \(H\).

\begin{lemma}[Codegree core]\label{lem:codegree-core}
Let \(H=(A,B;E)\) be a bipartite internally strongly \(C_6\)-connected graph.  Then
\[
        |N_H(a)\cap N_H(a')|\ge2
\]
for every pair of distinct active vertices \(a,a'\in A\), and
\[
        |N_H(b)\cap N_H(b')|\ge2
\]
for every pair of distinct active vertices \(b,b'\in B\).
\end{lemma}

\begin{proof}
We prove the assertion for \(A\); the proof for \(B\) is identical.

First suppose that distinct active vertices \(a,a'\in A\) have no common neighbour in \(H\).  Choose incident edges \(ab\) and \(a'b'\) of \(H\).  In a bipartite \(C_4\) or \(C_6\), any two vertices in the same bipartition class have a common neighbour on the cycle.  Hence no \(C_4\) or \(C_6\) contained in \(H\) can contain both edges \(ab\) and \(a'b'\), contradicting internal strong \(C_6\)-connectedness.

Now suppose that \(a\) and \(a'\) have exactly one common neighbour \(b\) in \(H\).  The adjacent edges \(ab\) and \(a'b\) cannot lie together on a \(C_4\) in \(H\), since such a \(C_4\) would give a second common neighbour of \(a\) and \(a'\).  This contradicts the adjacent-edge condition.
\end{proof}

\begin{lemma}[Fibre uniqueness]\label{lem:fibre-unique}
Every internally strongly \(C_6\)-connected subgraph \(H\subseteq G_C\) has at most one active vertex in each left fibre
\[
        \{x_{i,a}:a\in\Z_t\},
\]
and at most one active vertex in each right fibre
\[
        \{y_{j,b}:b\in\Z_t\}.
\]
\end{lemma}

\begin{proof}
Take two distinct vertices \(x_{i,u},x_{i,v}\) in the same left fibre, with \(u\ne v\).  For every right fibre \(j\), their neighbours in that fibre are
\[
        y_{j,u+c_{ij}}
        \qquad\text{and}\qquad
        y_{j,v+c_{ij}},
\]
which are distinct.  Thus \(x_{i,u}\) and \(x_{i,v}\) have no common neighbour in \(G_C\), and hence no common neighbour in any subgraph of \(G_C\).  By \cref{lem:codegree-core}, they cannot both be active in an internally strongly \(C_6\)-connected subgraph.

The right-fibre argument is the same: if \(y_{j,u}\) and \(y_{j,v}\) are distinct vertices in the same right fibre, then their neighbours in a left fibre \(i\) are \(x_{i,u-c_{ij}}\) and \(x_{i,v-c_{ij}}\), which are distinct.
\end{proof}

\section{Aligned two-covers}\label{sec:aligned-covers}

Fix \(I\subseteq[q]\) and shifts \(\alpha_i\in\Z_t\) for \(i\in I\).  For each column \(j\in[q]\) and residue \(h\in\Z_t\), define
\[
        S_{j,h}=\{i\in I:c_{ij}+\alpha_i=h\}.
\]

\begin{definition}[Aligned row two-cover]\label{def:row-cover}
An \emph{aligned row two-cover of size \(a\)} is a triple
\[
        \bigl(I,(\alpha_i)_{i\in I},(h_j)_{j\in[q]}\bigr),
\]
where \(I\subseteq[q]\), \(|I|=a\), \(\alpha_i\in\Z_t\) for \(i\in I\), and \(h_j\in\Z_t\) for \(j\in[q]\), such that every pair of distinct rows \(i,i'\in I\) is contained together in at least two chosen sets:
\[
        \#\{j\in[q]: i,i'\in S_{j,h_j}\}\ge2.
\]
\end{definition}

The column analogue uses the sign arising from the edge equation.  Fix \(J\subseteq[q]\) and shifts \(\gamma_j\in\Z_t\) for \(j\in J\).  For each row \(i\in[q]\) and residue \(h\in\Z_t\), define
\[
        T_{i,h}=\{j\in J:c_{ij}-\gamma_j=h\}.
\]

\begin{definition}[Aligned column two-cover]\label{def:column-cover}
An \emph{aligned column two-cover of size \(a\)} is a triple
\[
        \bigl(J,(\gamma_j)_{j\in J},(h_i)_{i\in[q]}\bigr),
\]
where \(J\subseteq[q]\), \(|J|=a\), \(\gamma_j\in\Z_t\) for \(j\in J\), and \(h_i\in\Z_t\) for \(i\in[q]\), such that every pair of distinct columns \(j,j'\in J\) is contained together in at least two chosen sets:
\[
        \#\{i\in[q]: j,j'\in T_{i,h_i}\}\ge2.
\]
\end{definition}

The sign convention matches
\[
        x_{i,\alpha_i}y_{j,\gamma_j}\in E(G_C)
        \quad\Longleftrightarrow\quad
        \gamma_j=\alpha_i+c_{ij}
        \quad\Longleftrightarrow\quad
        c_{ij}-\gamma_j=-\alpha_i.
\]

\begin{remark}[Monotonicity]\label{rem:monotonicity}
The aligned two-cover condition is inherited by subsets.  Hence, if a matrix has no aligned row two-cover of size \(a\), then it has no aligned row two-cover of any size at least \(a\).  The same statement holds for aligned column two-covers.
\end{remark}

\section{The occupancy-tail lemma}\label{sec:occupancy}

Let \((s,\kappa)\) be admissible.  Let \(C\) be uniformly random in \(\Z_t^{q\times q}\), with all entries independent, where \(q=\floor{\kappa t^s}\).  Set
\[
        a=\ceil{\frac{q}{\sqrt t\log t}}.
\]
Then
\[
        \frac aq=\frac{1+o(1)}{\sqrt t\log t},
\]
so, for all sufficiently large \(t\), one has \(1\le a<q\).

\begin{lemma}[Occupancy tail]\label{lem:occupancy}
With probability tending to \(1\), the matrix \(C\) has no aligned row two-cover of size \(a\) and no aligned column two-cover of size \(a\).
\end{lemma}

\begin{proof}
It suffices to prove the row statement.  The column statement follows by applying the same argument to the transpose of \(C\), with shifts \(-\gamma_j\), since
\[
        (C^T)_{j,i}+(-\gamma_j)=c_{ij}-\gamma_j.
\]

Fix a row set \(I\subseteq[q]\), \(|I|=a\), and fixed shifts \(\alpha_i\in\Z_t\) for \(i\in I\).  For each column \(j\), define
\[
        M_j=\max_{h\in\Z_t}|\{i\in I:c_{ij}+\alpha_i=h\}|.
\]
For fixed \(I\) and \(\alpha\), the values \(c_{ij}+\alpha_i\), \(i\in I\), are independent uniform elements of \(\Z_t\), and the random variables \(M_j\) are independent as \(j\) varies.

If an aligned row two-cover exists for this fixed \(I\) and \(\alpha\), then for some residues \(h_j\),
\[
        \sum_{j=1}^q\binom{|S_{j,h_j}|}{2}
        \ge2\binom a2
        =a(a-1).
\]
Since \(|S_{j,h_j}|\le M_j\), this implies
\[
        \sum_{j=1}^qM_j^2\ge2a(a-1).
\]
Put
\[
        R=\frac{a}{100\sqrt q}.
\]
Because \(q=\floor{\kappa t^s}\) and \(s>1\),
\[
        a=\frac{q}{\sqrt t\log t}(1+o(1))
        =\Theta_{s,\kappa}\left(\frac{t^{s-1/2}}{\log t}\right),
\]
and
\[
        R=\Theta_{s,\kappa}\left(\frac{t^{(s-1)/2}}{\log t}\right)\to\infty.
\]
The columns with \(M_j<R\) contribute less than
\[
        qR^2=\frac{a^2}{10^4}
\]
to \(\sum_jM_j^2\).  Therefore, on the fixed-cover event,
\[
        \sum_{j:M_j\ge R}M_j^2
        \ge2a(a-1)-\frac{a^2}{10^4}
        \ge1.8a^2
\]
for all sufficiently large \(t\).  Define
\[
        Y_j=M_j^2\mathbf1_{\{M_j\ge R\}}.
\]
Thus the fixed-cover event implies
\[
        \sum_{j=1}^qY_j\ge1.8a^2.
\]

For one column and \(1\le r\le a\), the standard balls-in-bins union bound gives
\[
        \Pr(M_j\ge r)
        \le t\binom ar t^{-r}
        \le t\left(\frac{ea}{rt}\right)^r.
\]
Indeed, choose the common residue, choose \(r\) rows, and require those \(r\) entries to equal that residue.  Let
\[
        \lambda=\frac{0.9\log t}{a}.
\]
For \(\ceil{R}\le r\le a\), write \(r=xa\).  Then
\[
        \frac{1}{100\sqrt q}\le x\le1.
\]
The tail bound gives
\begin{align*}
        \Pr(M_j\ge r)e^{\lambda r^2}
        &\le
        \exp\left(\log t-r\log\frac{rt}{ea}+\lambda r^2\right) \\
        &=\exp\left(\log t-r\left(\log\frac{xt}{e}-0.9x\log t\right)\right).
\end{align*}
Set
\[
        \phi_t(x)=\log\frac{xt}{e}-0.9x\log t.
\]
The function \(\phi_t\) is concave on \((0,\infty)\), so its minimum on \([1/(100\sqrt q),1]\) is attained at an endpoint.  At \(x=1\),
\[
        \phi_t(1)=0.1\log t-1.
\]
At \(x=1/(100\sqrt q)\),
\[
        \phi_t(x)=\log\frac{t}{100e\sqrt q}-0.9x\log t.
\]
If \(1<s<2\), then
\[
        \phi_t\left(\frac1{100\sqrt q}\right)
        =\left(1-\frac s2\right)\log t-O_{s,\kappa}(1)-o(1),
\]
and therefore there is a constant \(\xi=\xi(s,\kappa)>0\) such that, for all large \(t\),
\[
        \phi_t(x)\ge\xi\log t
        \qquad\left(\frac1{100\sqrt q}\le x\le1\right).
\]
If \(s=2\), then the admissibility condition \(\kappa<(100e)^{-2}\) gives
\[
        D_\kappa=\log\frac1{100e\sqrt\kappa}>0,
\]
and
\[
        \phi_t\left(\frac1{100\sqrt q}\right)=D_\kappa+o(1).
\]
Taking \(\xi=D_\kappa/2\), and using also \(\phi_t(1)\to\infty\), for all large \(t\) we have
\[
        \phi_t(x)\ge\xi
        \qquad\left(\frac1{100\sqrt q}\le x\le1\right).
\]

Consequently, if \(L_t=\log t\) in the case \(1<s<2\) and \(L_t=1\) in the case \(s=2\), then in both cases
\[
        \Pr(M_j\ge r)e^{\lambda r^2}
        \le \exp(\log t-\xi rL_t)
        \qquad(\ceil{R}\le r\le a).
\]
Summing over \(r\),
\[
        \sum_{r=\ceil{R}}^a\Pr(M_j\ge r)e^{\lambda r^2}
        \le a\exp(\log t-\xi RL_t).
\]
This upper bound is \(o(1/q)\).  Indeed, if \(1<s<2\), then
\[
        RL_t=\Theta_{s,\kappa}\bigl(t^{(s-1)/2}\bigr),
\]
so \(aq\exp(\log t-\xi RL_t)\to0\).  If \(s=2\), then
\[
        RL_t=R=\Theta_\kappa\left(\frac{\sqrt t}{\log t}\right),
\]
and again \(aq\exp(\log t-\xi RL_t)\to0\).  Hence
\[
        \mathbb E e^{\lambda Y_j}
        \le1+\sum_{r=\ceil{R}}^a\Pr(M_j=r)e^{\lambda r^2}
        \le1+o(1/q),
\]
uniformly over the fixed choice of \(I\) and \(\alpha\).  The variables \(Y_j\) are independent, so
\[
        \mathbb E\exp\left(\lambda\sum_{j=1}^qY_j\right)
        \le(1+o(1/q))^q
        =1+o(1).
\]
By Markov's inequality,
\[
        \Pr\left(\sum_{j=1}^qY_j\ge1.8a^2\right)
        \le(1+o(1))e^{-1.8\lambda a^2}
        =\exp\bigl(-(1.62+o(1))a\log t\bigr).
\]
Thus, for this fixed \(I\) and \(\alpha\), the probability that an aligned row two-cover exists is at most
\[
        \exp\bigl(-(1.62+o(1))a\log t\bigr).
\]

There are at most
\[
        \binom qa t^a\le\left(\frac{eq}{a}\right)^a t^a
\]
choices of \(I\) and \(\alpha\).  Since
\[
        \frac qa=(1+o(1))\sqrt t\log t,
\]
we have
\begin{align*}
        \log\left(\binom qa t^a\right)
        &\le a\log\frac{eq}{a}+a\log t \\
        &=a\left(\frac12\log t+\log\log t+O(1)\right)+a\log t \\
        &=(1.5+o(1))a\log t.
\end{align*}
The union bound gives
\[
        \Pr(\text{there is an aligned row two-cover of size }a)
        \le\exp\bigl(-(0.12+o(1))a\log t\bigr)=o(1).
\]
Applying the same argument to \(C^T\) gives the column statement.  A final union bound completes the proof.
\end{proof}

\section{Bounding internally strongly \texorpdfstring{$C_6$}{C6}-connected subgraphs}\label{sec:negative-proof}

\begin{lemma}[Active fibres give aligned covers]\label{lem:active-fibres-cover}
Let \(a\) be a positive integer, let \(C\in\Z_t^{q\times q}\), and let \(H\subseteq G_C\) be internally strongly \(C_6\)-connected.  If \(H\) has at least \(a\) active left fibres, then \(C\) has an aligned row two-cover of size \(a\).  If \(H\) has at least \(a\) active right fibres, then \(C\) has an aligned column two-cover of size \(a\).
\end{lemma}

\begin{proof}
We prove the left-fibre assertion.  Choose \(a\) active left fibres.  Let their base indices form \(I\), and write the active left vertex in fibre \(i\in I\) as \(x_{i,\alpha_i}\).  By \cref{lem:fibre-unique}, every active right fibre contains at most one active vertex.  For each active right fibre \(j\), write its active vertex as \(y_{j,\gamma_j}\) and set \(h_j=\gamma_j\); choose \(h_j\) arbitrarily for inactive right fibres.

Take distinct \(i,i'\in I\).  By \cref{lem:codegree-core}, \(x_{i,\alpha_i}\) and \(x_{i',\alpha_{i'}}\) have at least two common neighbours in \(H\).  These common neighbours are active and, by \cref{lem:fibre-unique}, lie in distinct right fibres.  If \(y_{j,\gamma_j}\) is such a common neighbour, then
\[
        c_{ij}+\alpha_i=\gamma_j=h_j,
        \qquad
        c_{i'j}+\alpha_{i'}=\gamma_j=h_j.
\]
Thus \(i,i'\in S_{j,h_j}\).  Since every pair \(i,i'\) has at least two such right fibres, the chosen data form an aligned row two-cover of size \(a\).

The right-fibre assertion is the same argument with the sign convention in \cref{def:column-cover}.  Choose \(a\) active right fibres with base indices \(J\), and write their active vertices as \(y_{j,\gamma_j}\) for \(j\in J\).  For each active left fibre \(i\), write its active vertex as \(x_{i,\alpha_i}\) and set \(h_i=-\alpha_i\); choose \(h_i\) arbitrarily for inactive left fibres.  A common active left neighbour of \(y_{j,\gamma_j}\) and \(y_{j',\gamma_{j'}}\) satisfies
\[
        c_{ij}-\gamma_j=-\alpha_i=h_i,
        \qquad
        c_{ij'}-\gamma_{j'}=-\alpha_i=h_i,
\]
which is precisely the aligned column condition.
\end{proof}

\begin{proposition}[Fixed lift parameter]\label{prop:fixed-s}
Let \((s,\kappa)\) be admissible, put \(q=\floor{\kappa t^s}\), and let
\[
        a=\ceil{\frac{q}{\sqrt t\log t}}.
\]
For all sufficiently large \(t\), there exists a matrix \(C\in\Z_t^{q\times q}\) such that every internally strongly \(C_6\)-connected subgraph \(H\subseteq G_C\) satisfies
\[
        e(H)<a^2
        =O_{s,\kappa}\left(\frac{q^2}{t(\log t)^2}\right)
        =O_{s,\kappa}\left(\frac{\rho(G_C)^3n^2}{(\log t)^2}\right).
\]
\end{proposition}

\begin{proof}
By \cref{lem:occupancy}, for all sufficiently large \(t\) there is a matrix \(C\) with no aligned row two-cover of size \(a\) and no aligned column two-cover of size \(a\).  Fix such a matrix.

Let \(H\subseteq G_C\) be internally strongly \(C_6\)-connected.  By \cref{lem:fibre-unique}, \(H\) has at most one active vertex in each left fibre and at most one active vertex in each right fibre.  Let \(A_H\) be the number of active left fibres and \(B_H\) the number of active right fibres.

If \(A_H\ge a\), then \cref{lem:active-fibres-cover} gives an aligned row two-cover of size \(a\), a contradiction.  Hence \(A_H<a\).  Similarly, \(B_H<a\).  There is at most one possible edge between any active left fibre and any active right fibre, since each active fibre contains only one active vertex and \(G_C\) is simple.  Therefore
\[
        e(H)\le A_HB_H<a^2.
\]
The definition of \(a\) gives
\[
        a^2=O_{s,\kappa}\left(\frac{q^2}{t(\log t)^2}\right).
\]
Since \(\rho(G_C)^3n^2=q^2/(16t)\), this is the asserted bound.
\end{proof}

\begin{proof}[Proof of \cref{thm:negative-intro}]
Given \(\beta\in[1/3,1/2)\), choose
\[
        s=\frac1\beta-1.
\]
Then \(1<s\le2\).  If \(1<s<2\), set \(\kappa=1\).  If \(s=2\), set
\[
        \kappa=(200e)^{-2},
\]
which is below \((100e)^{-2}\).  Thus \((s,\kappa)\) is admissible.  Apply \cref{prop:fixed-s} with \(q=\floor{\kappa t^s}\), choosing one admissible matrix \(C\) for each sufficiently large \(t\).  The resulting graph has
\[
        n=2qt=\Theta_{s,\kappa}(t^{s+1}),
        \qquad
        e(G_C)=q^2t=\Theta_{s,\kappa}(t^{2s+1}).
\]
Since \(\beta=1/(s+1)\),
\[
        n^{2-\beta}=\Theta_{s,\kappa}\left(t^{(s+1)(2-1/(s+1))}\right)
        =\Theta_{s,\kappa}(t^{2s+1}),
\]
and therefore \(e(G_C)=\Theta_\beta(n^{2-\beta})\).  Also
\[
        \rho(G_C)=\frac1{4t}=\Theta_{s,\kappa}(t^{-1})=\Theta_\beta(n^{-\beta}).
\]
Moreover,
\[
        n^{2-3\beta}
        =\Theta_{s,\kappa}\left(t^{(s+1)(2-3/(s+1))}\right)
        =\Theta_{s,\kappa}(t^{2s-1})
        =\Theta_{s,\kappa}\left(\frac{q^2}{t}\right)
        =\Theta_{s,\kappa}(\rho(G_C)^3n^2).
\]
\Cref{prop:fixed-s} gives
\[
        e(H)=O_{s,\kappa}\left(\frac{q^2}{t(\log t)^2}\right)
\]
for every internally strongly \(C_6\)-connected subgraph \(H\subseteq G_C\).  Since \(\log n=\Theta_{s,\kappa}(\log t)\), this becomes
\[
        e(H)=O_\beta\left(\frac{n^{2-3\beta}}{(\log n)^2}\right)
        =O_\beta\left(\frac{\rho(G_C)^3n^2}{(\log n)^2}\right).
\]
This constructs the required infinite family.
\end{proof}

\section{A limitation of the random-lift obstruction}\label{sec:zero-slice}

The random-lift obstruction above is naturally tied to the threshold \(q\asymp t^2\), corresponding to \(\beta=1/3\).  Once \(q\) is larger than \(t^2\) by a logarithmic factor, the same random-lift model typically contains an internally strongly \(C_6\)-connected subgraph of the natural size \(\asymp q^2/t\).  This does not prove an internal lower bound for arbitrary graphs below \(1/3\); it only explains why this random-lift obstruction does not extend to fixed power scales \(q\asymp t^s\) with \(s>2\).

\begin{proposition}[Zero-slice subgraph]\label{prop:zero-slice}
Let \(C\) be uniformly random in \(\Z_t^{q\times q}\), and suppose that
\[
        \frac{q}{t^2\log q}\to\infty.
\]
Then, with probability tending to \(1\), the graph \(G_C\) contains an internally strongly \(C_6\)-connected subgraph with \((1+o(1))q^2/t\) edges.
\end{proposition}

\begin{proof}
Let
\[
        X_0=\{x_{i,0}:i\in[q]\},
        \qquad
        Y_0=\{y_{j,0}:j\in[q]\},
\]
and let \(F=G_C[X_0,Y_0]\).  For each pair \((i,j)\), the edge \(x_{i,0}y_{j,0}\) is present exactly when \(c_{ij}=0\).  Hence \(F\) has the distribution of the binomial random bipartite graph \(G(q,q,1/t)\).

The number of edges of \(F\) has mean \(q^2/t\), so Chernoff's inequality gives
\[
        e(F)=(1+o(1))\frac{q^2}{t}
\]
with probability tending to \(1\), since \(q^2/t\to\infty\).

For two distinct vertices on the same side of \(F\), their common neighbourhood size is binomial with parameters \(q\) and \(1/t^2\), and has mean
\[
        \mu=\frac q{t^2}.
\]
By assumption, \(\mu/\log q\to\infty\).  If \(X\sim\operatorname{Bin}(q,1/t^2)\), then
\begin{align*}
        \Pr(X<2)
        &=\Pr(X=0)+\Pr(X=1) \\
        &=(1-t^{-2})^q+qt^{-2}(1-t^{-2})^{q-1} \\
        &\le e^{-\mu}+\mu e^{-\mu+t^{-2}}
         =\exp(-(1+o(1))\mu)
         =o(q^{-2}).
\end{align*}
A union bound over all pairs of vertices on both sides shows that, with probability tending to \(1\), every two same-side vertices of \(F\) have at least two common neighbours.

On this event, \(F\) is internally strongly \(C_6\)-connected.  If two adjacent edges share a left endpoint \(x\), say \(xy\) and \(xy'\), then \(y\) and \(y'\) have at least two common neighbours on the left side; one is \(x\), so there is another, which gives a \(C_4\).  The case of a shared right endpoint is identical.

Now take two disjoint edges \(xy\) and \(x'y'\).  If both cross-edges \(xy'\) and \(x'y\) are present, then the two edges lie on the corresponding \(C_4\).  Otherwise, at most one of \(y,y'\) is a common neighbour of \(x,x'\).  Since \(x,x'\) have at least two common neighbours, there is a common neighbour \(z\notin\{y,y'\}\).  Similarly, there is a common neighbour \(w\notin\{x,x'\}\) of \(y,y'\).  The chosen vertices lie on the appropriate sides and avoid the endpoints just specified, so
\[
        x-y-w-y'-x'-z-x
\]
is a simple \(C_6\) containing the two original edges.  Thus every two distinct edges of \(F\) lie on a \(C_4\) or \(C_6\), and adjacent pairs lie on a \(C_4\).
\end{proof}

For the parametrization \(q\asymp t^s\), the hypothesis \(q/(t^2\log q)\to\infty\) holds for every fixed \(s>2\), corresponding to \(0<\beta<1/3\).

\section{Concluding remarks}

The internal core theorem and the random-lift obstruction identify the exponent \(1/3\) as a natural boundary for the methods in this paper.  The positive theorem gives internal \(C_{\le6}\)- and \(C_{\le8}\)-connected cores of the expected sizes \(\rho^3n^2\) and \(\rho^2n^2\) whenever \(\rho\ge n^{-1/3}\).  The negative theorem shows that adding the adjacent-edge \(C_4\) requirement to the internal \(C_6\) statement changes the problem: at every fixed scale \(\rho=\Theta(n^{-\beta})\), \(1/3\le\beta<1/2\), no constant multiple of \(\rho^3n^2\) can be forced for internally strongly \(C_6\)-connected subgraphs.

The ambient strong \(C_6\) companion in \cref{sec:ambient-c6} should be read with the witness convention in mind.  Its witness cycles may use edges outside the selected edge set, while all cycles in \cref{thm:internal-cores-intro,thm:negative-intro} are required to lie inside the subgraph under discussion.

\section*{Acknowledgements}
The author acknowledges the use of OpenAI's ChatGPT during the preparation of this manuscript. While it was used for ideation, formulation, proof exploration and refinement, narrowing the search space, programming, LaTeX formatting and other forms of orchestration, the author nonetheless takes full responsibility for the accuracy of the final contents of this paper.

\end{document}